\documentclass[11pt,a4paper]{amsart}
\usepackage{hyperref}
\hypersetup{	
	colorlinks,
	linkcolor=black,
	citecolor=black,
	urlcolor=black,
	breaklinks=true
}
\usepackage[applemac]{inputenc}
\usepackage[T1]{fontenc}
\usepackage[english]{babel}
\usepackage{
	amssymb,
	mathtools,
	todonotes
}
\usepackage{geometry}
\newtheoremstyle{lemma}{.5\baselineskip\@plus.2\baselineskip\@minus.2\baselineskip}{.5\baselineskip\@plus.2\baselineskip\@minus.2\baselineskip}
	{\itshape}
	{}
	{\bfseries}
	{.}
	{\newline}
	{\thmname{#1}\thmnumber{ #2}\thmnote{ (#3)}}	
\theoremstyle{lemma}
	\newtheorem{theorem}{Theorem}
	\newtheorem{lemma}[theorem]{Lemma}  
	\newtheorem{proposition}[theorem]{Proposition}
	\newtheorem{corollary}[theorem]{Corollary}

\newtheoremstyle{lemmanonumber}{.5\baselineskip\@plus.2\baselineskip\@minus.2\baselineskip}{.5\baselineskip\@plus.2\baselineskip\@minus.2\baselineskip}
	{\itshape}
	{}
	{\bfseries}
	{.}
	{\newline}
	{\thmname{#1}\thmnote{ (#3)}}
\theoremstyle{lemmanonumber}
	\newtheorem{proposition3}{Proposition 3}
	\newtheorem{proposition4}{Proposition 4}
	\newtheorem{proposition5}{Proposition 5}

\newtheoremstyle{definition}{.5\baselineskip\@plus.2\baselineskip\@minus.2\baselineskip}{.5\baselineskip\@plus.2\baselineskip\@minus.2\baselineskip}
	{}
	{}
	{\bfseries}
	{.}
	{\newline}
	{\thmname{#1}\thmnumber{ #2}\thmnote{ (#3)}}	
\theoremstyle{definition}
	\newtheorem{definition}[theorem]{Definition}

\newcommand{\N}{\ensuremath{\mathbb{N}}}

\newcommand{\R}{\ensuremath{\mathbb{R}}}

\newcommand{\sphere}{\ensuremath{\mathbb{S}}}

\DeclarePairedDelimiter\abs{\lvert}{\rvert}

\newcommand{\dd}{\ensuremath{\mathrm{d}}}

\DeclareMathOperator{\minRad}{minRad}
\DeclareMathOperator{\maxCurv}{maxCurv}
\DeclareMathOperator{\dcsd}{dcsd}
\DeclareMathOperator{\scsd}{scsd}
\DeclareMathOperator{\dcrit}{dcrit}
\DeclareMathOperator{\crit}{crit}

\renewcommand{\phi}{\varphi}
\renewcommand{\epsilon}{\varepsilon}

\hypersetup{
	pdftitle={Discrete Thickness},
	pdfsubject={},
	pdfauthor={Sebastian Scholtes},
	pdfkeywords={},
	pdfcreator={},
	pdfproducer={}
}

\makeindex
\begin{document}

\title{Discrete Thickness}
\author{\href{mailto:sebastian.scholtes@rwth-aachen.de}{Sebastian Scholtes}}
\address{Institut f. Mathematik, RWTH Aachen University, Templergraben 55, D-52062 Aachen, Germany}
\email{\href{mailto:sebastian.scholtes@rwth-aachen.de}{sebastian.scholtes@rwth-aachen.de}}
\urladdr{\href{http://www.instmath.rwth-aachen.de/~scholtes/home/}{http://www.instmath.rwth-aachen.de/~scholtes/home/}}
\date{\today}
\begin{abstract}
	We investigate the relationship between a discrete version of thickness and its smooth counterpart.	These discrete energies are defined on equilateral polygons with $n$ vertices. 
	It will turn out that the smooth ropelength, which is the scale invariant quotient of length divided by thickness, is the $\Gamma$-limit of the discrete ropelength for $n\to\infty$, 
	regarding the topology induced by the Sobolev norm $||\cdot||_{W^{1,\infty}(\sphere_{1},\R^{d})}$. This result directly implies the convergence of almost minimizers of the discrete 
	energies in a fixed knot class to minimizers of the smooth energy. Moreover, we show that the unique absolute minimizer of inverse discrete thickness is the regular $n$-gon.
\end{abstract}
\maketitle

\section{Introduction }

In this article we are concerned with the relationship of a discrete version of the \emph{thickness} $\Delta$ of a curve $\gamma$, defined by
\begin{align*}
	\Delta[\gamma]\vcentcolon=\inf_{\substack{x,y,z\in \gamma(\sphere_{1})\\x\not=y\not=z\not=x}} r(x,y,z)
\end{align*}
on $\mathcal{C}$, the set of all curves $\gamma:\sphere_{1}\to \R^{d}$ that are parametrised by arc length, i.e., $\gamma\in C^{0,1}(\sphere_{1},\R^{d})=W^{1,\infty}(\sphere_{1},\R^{d})$
with $|\gamma'|=1$ a.e., and have length $\int_{\sphere_{1}}\abs{\gamma'}\,\dd t=1$. Here, $\sphere_{1}$ is the circle of length $1$ and $r(x,y,z)$ the radius of the unique circle that 
contains $x,y$ and $z$, which is set to infinity if the three points are collinear. This notion of thickness was introduced in \cite{Gonzalez1999a} and is equivalent to the Federer's reach,
see \cite{Federer1959a}. Geometrically, the thickness of a curve gives the radius of the largest uniform tubular neighbourhood
about the curve that does not intersect itself. The \emph{ropelength}, which is length divided by thickness, is scale invariant and a knot is called \emph{ideal} if it minimizes ropelength in 
a fixed knot class or, equivalently, minimizes this energy amongst all curves in this knot class with fixed length. These ideal knots are of great interest, not only to mathematicians but also to biologists, 
chemists, physicists, $\ldots$, since they exposit interesting physical features and resemble the time-averaged shapes of knotted DNA molecules in solution \cite{Stasiak1996a,Katritch1996a,Katritch1997a}, 
see \cite{1998a,Simon2002a} for an overview of physical knot theory with applications. The existence of ideal knots in every knot class was settled
in \cite{Cantarella2002a,Gonzalez2002b,Gonzalez2003a} and it was found that the unique absolute minimizer (in all knot classes) is the round circle.
Furthermore, this energy is self-repulsive, meaning that finite energy prevents the curve from having self intersections. By now it is well-known that thick curves, or in general manifolds of
positive reach, are of class $C^{1,1}$ and vice versa, see \cite{Federer1959a,Schuricht2003a,Lytchak2005a,Scholtes2013a}.
In \cite{Cantarella2002a} it was shown that ideal links must not be of class $C^{2}$ and computer experiments in \cite{Sullivan2002a} suggest that
$C^{1,1}$ regularity is optimal for knots, too. Still, there is a conjecture \cite[Conjecture 24]{Cantarella2002a} that ropelength minimizers are piecewise analytic. Further interesting properties of critical
points for the ropelength as well as the Euler-Lagrange equation were derived in \cite{Schuricht2003a,Schuricht2004a,Cantarella2011a}.\\

Another way to write the thickness of an arc length curve is
\begin{align}\label{formulathickness2}
	\Delta[\gamma]=\min\{\minRad(\gamma),2^{-1}\dcsd(\gamma)\},
\end{align}
which by \cite{Schuricht2003a}
holds for all arc length curves with positive thickness. The \emph{minimal radius of curvature} $\minRad(\gamma)$ of $\gamma$
is the inverse of the \emph{maximal curvature} $\maxCurv(\gamma)\vcentcolon=||\gamma''||_{L^{\infty}}$ and $\dcsd(\gamma)\vcentcolon=\min_{(x,y)\in \dcrit(\gamma)}\abs{y-x}$ 
is the \emph{doubly critical selfdistance}. The set of \emph{doubly critical points} $\dcrit(\gamma)$ of a $C^{1}$ curve $\gamma$ consists of all
pairs $(x,y)$ where $x=\gamma(t)$ and $y=\gamma(s)$ are distinct points on $\gamma$ so that $\langle \gamma'(t),\gamma(t)-\gamma(s) \rangle=\langle \gamma'(s),\gamma(t)-\gamma(s) \rangle=0$,
i.e., $s$ is critical for $u\mapsto\abs{\gamma(t)-\gamma(u)}^{2}$ and and $t$ for $v\mapsto\abs{\gamma(v)-\gamma(s)}^{2}$.\\

Appropriate versions of thickness for polygons derived from the representation in \eqref{formulathickness2} are already available. The curvature of a polygon, localized at a vertex $y$,
is defined by
\begin{align*}
	\kappa_{d}(x,y,z)\vcentcolon=\frac{2\tan(\frac{\phi}{2})}{\frac{\abs{x-y}+\abs{z-y}}{2}}\quad\text{and as an alternative}\quad
	\kappa_{d,2}(x,y,z)\vcentcolon=\frac{\phi}{\frac{\abs{x-y}+\abs{z-y}}{2}}
\end{align*}
where $x$ and $z$ are the vertices adjacent to $y$ and $\phi=\measuredangle(y-x,z-y)$ is the exterior angle at $y$, note $\kappa_{d,2}\leq \kappa_{d}$.
We then set $\minRad(p)\vcentcolon=\maxCurv(p)^{-1}\vcentcolon=\min_{i=1,\ldots,n}\kappa_{d}^{-1}(x_{i-1},x_{i},x_{i+1})$ if the polygon $p$ has the consecutive vertices $x_{i}$, 
$x_{0}\vcentcolon= x_{n}$, $x_{n+1}\vcentcolon=x_{1}$;
$\minRad_{2}$ and $\maxCurv_{2}$ are defined accordingly. The doubly critical self distance of a polygon $p$ is given as for a smooth curve if we define $\dcrit(p)$ to consist of 
pairs $(x,y)$ where $x=p(t)$ and $y=p(s)$ and  $s$ locally extremizes $u\mapsto\abs{p(t)-p(u)}^{2}$ and $t$ locally extremizes $v\mapsto\abs{p(v)-p(s)}^{2}$.
Now, the discrete thickness $\Delta_{n}$ defined on $\mathcal{P}_{n}$, the class of arc length parametrisations of equilateral polygons of length $1$ with $n$ segments is defined analogous to \eqref{formulathickness2} by
\begin{align*}
	\Delta_{n}[p]=\min\{\minRad(p),2^{-1}\dcsd(p)\}
\end{align*}
if all vertices are distinct and $\Delta_{n}[p]=0$ if two vertices of $p$ coincide.
This notion of thickness was introduced and investigated by Rawdon in \cite{Rawdon1997a,Rawdon1998a,Rawdon2000a,Rawdon2003a} and by Millett, Piatek and Rawdon in \cite{Millett2008a}.
In this series of works alternative representations of smooth and discrete thickness were established that were then used to show that not only does the value of the minimal discrete inverse thickness
converge to the minimal smooth inverse thickness in every tame knot class, but, additionally, a subsequence of the discrete equilateral minimizers, which are shown to exist in every tame knot class, 
converge to a smooth minimizer of the same knot type in the $C^{0}$ topology as the number of segments increases, at least if we require that all discrete minimizers are bounded in $L^{\infty}$.
Furthermore, it was shown that discrete thickness is continuous, for example on the space of simple equilateral polygons with fixed segment length. In \cite{Dai2000a,Rawdon2003a} similar questions for
more general energy functions were considered.\\

In the present work we continue this line of thought and investigate the way in which the discrete thickness approximates smooth thickness in more detail. It will turn out that the right framework
is given by $\Gamma$-convergence. This notion of convergence that was invented by DeGiorgi is devised in such a way, as to allow the convergence of minimizers and even almost minimizers.
For the convenience of the reader we summarise the relevant facts on $\Gamma$-convergence in Section \ref{sectiongammaconvergence}.

\begin{theorem}[Convergence of discrete inverse to smooth inverse thickness]\label{gammaconvergence}
	For every tame knot class $\mathcal{K}$ holds
	\begin{align*}
		\Delta_{n}^{-1}\xrightarrow{\Gamma}\Delta^{-1}\quad\text{on }(\mathcal{C}(\mathcal{K}),||\cdot||_{W^{1,\infty}(\sphere_{1},\R^{3})}).
	\end{align*}
\end{theorem}

Here, the addition of a knot class $\mathcal{K}$ means that only knots of this particular knot class are considered. The functionals are extended by infinity outside their natural domain.
By the properties presented in Section \ref{sectiongammaconvergence} together with Proposition \ref{compactnessdcsd}, we obtain the following convergence result of 
polygonal ideal knots to smooth ideal knots improving the convergence in \cite[Theorem 8.5]{Rawdon2003a} from $C^{0}$ to $W^{1,\infty}=C^{0,1}$.

\begin{corollary}[Ideal polygonal knots converge to smooth ideal knots]\label{theoremconvergenceofminimizers}
	Let $\mathcal{K}$ be a tame knot class, $p_{n}\in\mathcal{P}_{n}(\mathcal{K})$ bounded in $L^{\infty}$ with $\abs{\inf_{\mathcal{P}_{n}(\mathcal{K})}\Delta_{n}^{-1}-\Delta_{n}^{-1}(p_{n})}\to 0$. 
	Then there is a subsequence
	\begin{align*}
		p_{n_{k}}\xrightarrow[k\to\infty]{W^{1,\infty}(\sphere_{1},\R^{3})}\gamma\in \mathcal{C}(\mathcal{K})\quad\text{with}\quad
		\Delta^{-1}[\gamma]=\inf_{\mathcal{C}(\mathcal{K})}\Delta^{-1}=\lim_{k\to\infty}\Delta_{n_{k}}[p_{n_{k}}].
	\end{align*}
\end{corollary}

The subsequent compactness result is proven via a version of Schur's Comparison Theorem (see Proposition \ref{schurtheorem}) that allows to compare polygons with circles.

\begin{proposition}[Compactness]\label{compactness}
	Let $p_{n}\in \mathcal{P}_{n}(\mathcal{K})$ bounded in $L^{\infty}$ with $\liminf_{n\to\infty}\maxCurv(p_{n})<\infty$. Then there is $\gamma\in C^{1,1}(\sphere_{1},\R^{d})$ and a subsequence
	\begin{align*}
		p_{n_{k}}\xrightarrow[k\to\infty]{W^{1,\infty}(\sphere_{1},\R^{d})} \gamma\in\mathcal{C} \quad\text{with}\quad \maxCurv(\gamma)\leq \liminf_{n\to\infty}\maxCurv(p_{n}).
	\end{align*}
\end{proposition}

This result is then used to show another compactness result that additionally guarantees that the limit curve belongs to the same knot class, if one assures that the doubly critical self distance
is bounded, too.

\begin{proposition}[Compactness II]\label{compactnessdcsd}
	Let  and $p_{n}\in \mathcal{P}_{n}(\mathcal{K})$ bounded in $L^{\infty}$ with $\liminf_{n\to\infty}\Delta_{n}[p_{n}]^{-1}<\infty$. 
	Then there is
	\begin{align*}
		\gamma\in \mathcal{C}(\mathcal{K})\cap C^{1,1}(\sphere_{1},\R^{d})\quad\text{with}\quad p_{n_{k}}\to \gamma\text{ in }W^{1,\infty}(\sphere_{1},\R^{d}).
	\end{align*}
\end{proposition}

If the knot class is not fixed the unique absolute minimizers of $\Delta_{n}^{-1}$ is the regular $n$-gon.

\begin{proposition}[Regular $n$-gon is unique minimizer of $\Delta_{n}^{-1}$]
	Let $p\in \mathcal{P}_{n}$ and $g_{n}$ the regular $n$-gon. Then
	\begin{align*}
		\Delta_{n}[g_{n}]^{-1}\leq \Delta_{n}[p]^{-1},
	\end{align*}
	with equality if and only if $p$ is a regular $n$-gon.
\end{proposition}

\textbf{Acknowledgement }
The author thanks H. von der Mosel, for his interest and many useful suggestions and remarks.

\section{Prelude in $\Gamma$-convergence }\label{sectiongammaconvergence}

In this section we want to acquaint the reader with $\Gamma$-convergence and repeat its (to us) most important property.

\begin{definition}[$\,\Gamma$-convergence]
	Let $X$ be a topological space, $\mathcal{F},\mathcal{F}_{n}:X\to\overline \R\vcentcolon=\R\cup\{\pm\infty\}$. We say that \emph{$\mathcal{F}_{n}$ $\Gamma$-converges to $\mathcal{F}$}, 
	written $\mathcal{F}_{n}\stackrel{\Gamma}{\to}\mathcal{F}$, if
	\begin{itemize}
		\item
			for every $x_{n}\to x$ holds $\mathcal{F}(x)\leq \liminf_{n\to\infty}\mathcal{F}_{n}(x_{n}$),
		\item
			for every $x\in X$ there are $x_{n}\to x$ with $\limsup_{n\to\infty}\mathcal{F}_{n}(x_{n})\leq \mathcal{F}(x)$.
	\end{itemize}
\end{definition}

The first inequality is usually called \emph{$\liminf$ inequality} and the second one \emph{$\limsup$ inequality}. Note, that if the functionals are only defined on subspaces $Y$ and $Y_{n}$
of $X$ and we extend the functionals by plus infinity on the rest of $X$ it is enough to show the $\liminf$ inequality holds for every $x_{n}\in Y_{n}$, $x\in X$ 
and the $\limsup$ inequality for $x\in Y$ and $x_{n}\in Y_{n}$ in order to establish $\mathcal{F}_{n}\stackrel{\Gamma}{\to}\mathcal{F}$. 
In our application we have $X=\mathcal{C}(\mathcal{K})$, $Y=\mathcal{C}(\mathcal{K})\cap C^{1,1}(\sphere_{1},\R^{d})$ and $Y_{n}=\mathcal{P}_{n}(\mathcal{K})$.\\

This convergence is modeled in such a way that it allows the convergence of minimizers and even almost minimizers of the functionals $\mathcal{F}_{n}$ to minimizers of the limit functional $\mathcal{F}$.

\begin{theorem}[Convergence of minimizers, {\cite[Corollary 7.17, p.78]{Dal-Maso1993a}}]\label{dalmasoproposition}
	Let $\mathcal{F}_{n},\mathcal{F}:X\to\overline \R$ with $\mathcal{F}_{n}\stackrel{\Gamma}{\to}\mathcal{F}$. Let $\epsilon_{n}>0$, $\epsilon_{n}\to 0$ and $x_{n}\in X$ 
	with $|\inf \mathcal{F}_{n}-\mathcal{F}_{n}(x_{n})|\leq \epsilon_{n}$. If $x_{n_{k}}\to x$ then
	\begin{align*}
		\mathcal{F}(x)=\inf \mathcal{F}=\lim_{k\to\infty}\mathcal{F}_{n}(x_{n_{k}}).
	\end{align*}
\end{theorem}

In order to use this result in our application where we want to show that minimizers of the discrete functional $\mathcal{F}_{n}$ converge to minimizers of the ``smooth'' functional $\mathcal{F}$
we do need $\mathcal{F}_{n}\stackrel{\Gamma}{\to}\mathcal{F}$ as well as an additional compactness result that show that there is subsequence $x_{n_{k}}\to x$ with $x\in X$.

\section{Schur's Theorem for polygons }\label{sectionSchur}

In this section we want to estimate for how many vertices a polygon that starts tangentially at a sphere stays out of this sphere if the curvature of the polygon is bounded in terms of the radius of the sphere.
It turns out that make such an estimate we need Schur's Comparison Theorem for a polygon and a circle. This theorem for smooth curves basically says that if the curvature of a smooth curve 
is strictly smaller than the curvature of a convex planar curve then the endpoint distance of the planar convex curve is strictly smaller than the endpoint distance of the other curve.
There already is a version of this theorem for
classes of curves including polygons, see \cite[Theorem 5.1]{Sullivan2008a}, however, with the drawback that the hypotheses there do not allow to compare polygons and smooth curves.

\begin{proposition}[Schur's Comparison Theorem]\label{schurtheorem}
	Let $p\in C^{0,1}(I,\R^{d})$, $I=[0,L]$ be the arc length parametrisation of a polygon with $\maxCurv_{2}(p)\leq K$ and $KL\leq \pi$. Let $\eta$ be the arc length parametrisation of a
	circle of curvature $K$. Then
	\begin{align*}
		\abs{\eta(L)-\eta(0)}< \abs{p(L)-p(0)}.
	\end{align*}
\end{proposition}
\begin{proof}
	Let $p(a_{k})$ be the vertices of the polygon, $a_{0}=0$.
	We write $\alpha_{i,j}\vcentcolon=\measuredangle(p'(t_{i}),p'(t_{j}))$, where $t_{k}$ is an interior point of $I_{k}\vcentcolon=[a_{k-1},a_{k}]$.
	From the curvature bound we get $\alpha_{i,i+1}\leq K\frac{\abs{I_{i}}+\abs{I_{i+1}}}{2}$ and hence for $i\leq j$ we can estimate
	$\alpha_{i,j}\leq \sum_{k=i}^{j-1}\alpha_{k,k+1}\leq \frac{K}{2}\sum_{k=i}^{j-1}(\abs{I_{k}}+\abs{I_{k+1}})$. Now,
	\begin{align*}
		\MoveEqLeft
		\abs{p(L)-p(0)}^{2}=\int_{I}\int_{I}\langle p'(s),p'(u) \rangle\,\dd s\,\dd u
		=\sum_{i,j=1}^{n}\int_{I_{i}}\int_{I_{j}}\cos(\alpha_{i,j})\\
		&=\sum_{\substack{i,j=1\\i=j}}^{n}\abs{I_{i}}\abs{I_{j}}
		+2\sum_{\substack{i,j=1\\i<j}}^{n}\abs{I_{i}}\abs{I_{j}}\cos(\alpha_{i,j}).
	\end{align*}
	\begin{figure}
		\centering 
		\includegraphics[width=0.5\textwidth]{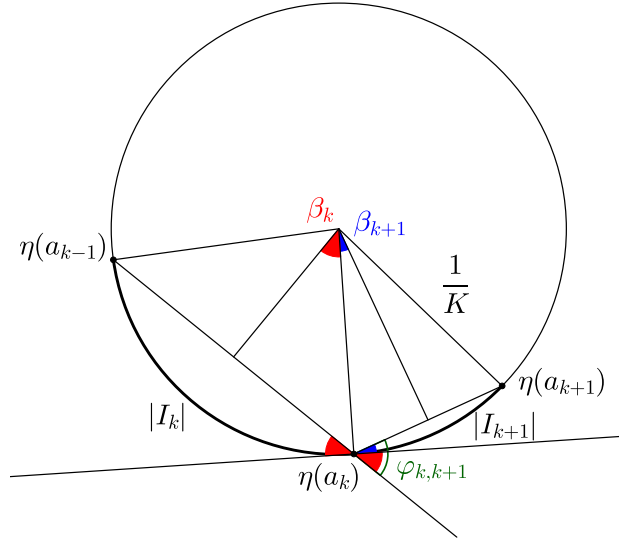}
		\caption{The angle $\phi_{k,k+1}$ in the proof of Proposition \ref{schurtheorem}.}
		\label{anglesinsribedcircle}
	\end{figure}
	Similarly,
	\begin{align*}
		\MoveEqLeft
		\abs{\eta(L)-\eta(0)}^{2}=\int_{I}\int_{I}\langle \eta'(s),\eta'(u) \rangle\,\dd s\,\dd u\\
		&=\sum_{\substack{i,j=1\\i=j}}^{n}\int_{I_{i}}\int_{I_{j}}\langle \eta'(s),\eta'(u) \rangle\,\dd s\,\dd u
		+2\sum_{\substack{i,j=1\\i<j}}^{n}\int_{I_{i}}\int_{I_{j}}\langle \eta'(s),\eta'(u) \rangle\,\dd s\,\dd u\\
		&\leq \sum_{\substack{i,j=1\\i=j}}^{n}\abs{I_{i}}\abs{I_{i}}
		+2\sum_{\substack{i,j=1\\i<j}}^{n}\langle \eta(a_{j})-\eta(a_{j-1}),\eta(a_{i})-\eta(a_{i-1}) \rangle.
	\end{align*}
	Write $\varphi_{i,j}\vcentcolon=\measuredangle(\eta(a_{j})-\eta(a_{j-1}),\eta(a_{i})-\eta(a_{i-1}))$. Then $\varphi_{i,j}=\sum_{k=i}^{j-1}\varphi_{k,k+1}$, because the points
	$\eta(a_{i})$ form a convex plane polygon. From Figure \ref{anglesinsribedcircle} we see that $\varphi_{k,k+1}=K\frac{\abs{I_{k}}+\abs{I_{k+1}}}{2}$ and hence $\alpha_{i,j}\leq\varphi_{i,j}$.
	This allows us to continue our estimate
	\begin{align*}
		\MoveEqLeft
		\abs{\eta(L)-\eta(0)}^{2}
		\leq\sum_{\substack{i,j=1\\i=j}}^{n}\abs{I_{i}}\abs{I_{i}}
		+2\sum_{\substack{i,j=1\\i<j}}^{n} \abs{\eta(a_{j})-\eta(a_{j-1})}\abs{\eta(a_{i})-\eta(a_{i-1})}\cos(\varphi_{i,j})\\
		&<\sum_{\substack{i,j=1\\i=j}}^{n}\abs{I_{i}}\abs{I_{i}}
		+2\sum_{\substack{i,j=1\\i<j}}^{n} \abs{I_{i}}\abs{I_{j}}\cos(\varphi_{i,j})
		\leq \sum_{\substack{i,j=1\\i=j}}^{n}\abs{I_{i}}\abs{I_{j}}
		+2\sum_{\substack{i,j=1\\i<j}}^{n}\abs{I_{i}}\abs{I_{j}}\cos(\alpha_{i,j})
		=\abs{p(L)-p(0)}^{2}.
	\end{align*}
\end{proof}

As we only need $\varphi_{1,n}=2^{-1}K\sum_{i=1}^{n-1}(\abs{I_{i}}+\abs{I_{i+1}})\leq \pi$ we can make do with $KL\leq \pi+2^{-1}K(\abs{I_{1}}+\abs{I_{n}})$ instead of $KL\leq \pi$.

\begin{corollary}[Tangential polygon stays outside of sphere]\label{polygonstaysoutsideofsphere}
	Let $p$ be an equilateral polygon of length $L$ with $\maxCurv_{2}(p)\leq K$ and $KL\leq \frac{\pi}{2}$.
	If $p$ touches a sphere of curvature $K$ at an endpoint then all other vertices of $p$ lie outside the sphere.
\end{corollary}
\begin{proof}
	\begin{figure}
		\centering 
		\includegraphics[width=0.3\textwidth]{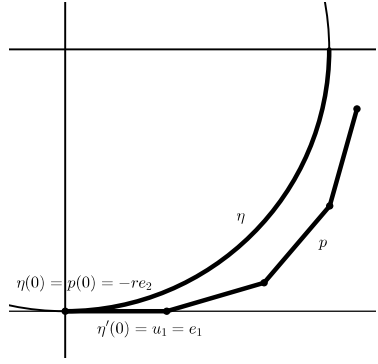}
		\caption{The situation in the proof of Corollary \ref{polygonstaysoutsideofsphere}.}
		\label{picturepolygonstaysoutsideofsphere}
	\end{figure}
	Without loss of generality we might assume that the sphere is centred at the origin and that $p$ touches the sphere at $p(0)=-re_{2}$ with $u_{1}=e_{1}$, where $r=K^{-1}$
	and $u_{i}\in \sphere^{d-1}$ are the directions of the edges.
	We have to show that $\abs{p(a_{k})}> r$ for $k=1,\ldots, n$. Let $\eta$ be the arc length parametrisation of the circle of radius $r$ about the origin in the $e_{1},e_{2}$ plane, starting 
	at $\eta(0)=p(0)$ with $\eta'(0)=u_{1}=e_{1}$.
	On the unit sphere equipped with the great circle distance, i.e., angle, we have $\frac{\pi}{2}=d(e_{1},e_{2})\leq d(e_{1},u_{1})+\sum_{i=1}^{k-1}d(u_{i},u_{i+1})+ d(u_{k},e_{2})$ 
	and hence $u_{1}=e_{1}$ and the curvature bound imply
	\begin{align*}
		\MoveEqLeft
		d(\eta'(a_{k-1}),e_{2})=d(e_{1},e_{2})-d(e_{1},\eta'(a_{k-1}))=\frac{\pi}{2}-d(\eta'(0),\eta'(a_{k-1}))=\frac{\pi}{2}-\int_{0}^{a_{k-1}}\abs{\eta''}\,\dd t\\
		&=\frac{\pi}{2}-Ka_{k-1}=\frac{\pi}{2}-K\sum_{i=1}^{k-1}\frac{\abs{I_{i}}+\abs{I_{i+1}}}{2}\leq \frac{\pi}{2}-\sum_{i=1}^{k-1}d(u_{i},u_{i+1})\leq d(u_{k},e_{2}),
	\end{align*}
	since $\eta'\vert_{[0,L]}$ is a parametrisation of the unit circle in the $e_{1},e_{2}$ plane from $e_{1}$ to $e_{2}$ with constant speed $\abs{\eta''}=K$.
	Now, we can estimate
	\begin{align}\label{estimateanglestartsecant}
		\begin{split}
			\MoveEqLeft
			\big\langle p(a_{k})-p(0),p(0)\big\rangle = \Big\langle \sum_{i=1}^{k}\abs{I_{i}}u_{i},-re_{2}\Big\rangle
			=-r\sum_{i=1}^{k} \abs{I_{i}} \cos(d(u_{i},e_{2})\\
			&\geq -r\sum_{i=1}^{k} \abs{I_{i}} \cos(d(\eta'(a_{i-1}),e_{2})
			\geq -r\sum_{i=1}^{k} \int_{I_{i}}\cos(d(\eta'(t),e_{2}))\,\dd t\\
			&=\int_{0}^{a_{k}}\langle \eta'(t),-re_{2}\rangle\,\dd t
			=\big\langle \eta(a_{k})-\eta(0),\eta(0)\big\rangle,
		\end{split}
	\end{align}
	as $d(\eta'(t),e_{2})\leq d(\eta'(a_{i-1}),e_{2})$ for $t\in I_{i}$. Using Schur's Comparison Theorem, Proposition \ref{schurtheorem}, and \eqref{estimateanglestartsecant} we conclude
	\begin{align*}
		\MoveEqLeft
		\abs{p(a_{k})}^{2}=\abs{p(a_{k})-p(0)+p(0)}^{2}
		=\abs{p(a_{k})-p(0)}^{2}+2\langle p(a_{k})-p(0),p(0)\rangle +\abs{p(0)}^{2}\\
		&> \abs{\eta(a_{k})-\eta(0)}^{2}+2\langle \eta(a_{k})-\eta(0),\eta(0)\rangle +\abs{\eta(0)}^{2}
		=\abs{\eta(a_{k})}^{2}=r^{2}.
	\end{align*}
\end{proof}

\section{Compactness }\label{sectioncompactness}

Note, that since the domain is bounded we have $C^{0,1}(\sphere_{1},\R^{d})=W^{1,\infty}(\sphere_{1},\R^{d})$.

\begin{proposition3}[Compactness]
	Let $p_{n}\in \mathcal{P}_{n}(\mathcal{K})$ bounded in $L^{\infty}$ with $\liminf_{n\to\infty}\maxCurv(p_{n})<\infty$. Then there is $\gamma\in C^{1,1}(\sphere_{1},\R^{d})$ and a subsequence
	\begin{align*}
		p_{n_{k}}\xrightarrow[k\to\infty]{W^{1,\infty}(\sphere_{1},\R^{d})} \gamma\in\mathcal{C} \quad\text{with}\quad \maxCurv(\gamma)\leq \liminf_{n\to\infty}\maxCurv(p_{n}).
	\end{align*}
\end{proposition3}
\begin{proof}
	\textbf{Step 1}
		Without loss of generality, by taking subsequences if necessary, we might assume $\maxCurv(p_{n})\leq K<\infty$ for all $n\in\N$. As $p_{n}$ is bounded in $W^{1,\infty}$ there is a subsequence 
		(for notational convenience denoted by the same indices) converging to $\gamma\in W^{1,2}(\sphere_{1},\R^{d})$ strongly in $C^{0}(\sphere_{1},\R^{d})$ and weakly in $W^{1,2}(\sphere_{1},\R^{d})$.
		First we have to show that $\gamma$ is also parametrised by arc length, i.e., $\abs{\gamma'}=1$ a.e.. Since
		$\abs{p_{n}'}=1$ a.e. testing with $\varphi=\gamma'\cdot\chi_{\{\abs{\gamma'}>1\}}$, $\chi_{A}$ the characteristic function of $A$, yields
		\begin{align*}
			\MoveEqLeft
			0\leftarrow
			\int_{\sphere_{1}}\langle p_{n}'-\gamma',\varphi \rangle\,\dd t
			=\int_{\{\abs{\gamma'}>1\}}\langle p_{n}'-\gamma',\gamma' \rangle\,\dd t\\
			&\leq \int_{\{\abs{\gamma'}>1\}}(\abs{p_{n}'}\abs{\gamma'}-\abs{\gamma'}^{2})\,\dd t
			= \int_{\{\abs{\gamma'}>1\}}\abs{\gamma'}\underbrace{(1-\abs{\gamma'})}_{<0}\,\dd t
		\end{align*}
		and thus $\abs{\gamma'}\leq 1=\abs{p_{n}'}$ a.e.. Additionally, we know from Schur's Theorem, Proposition \ref{schurtheorem}, that if $\eta$ is the arc length parametrisation of a circle of curvature $K$, 
		then for a.e. $t$ holds
		\begin{align*}
			\MoveEqLeft
			\abs{\gamma'(t)}=\lim_{h\to 0}\Big|\frac{\gamma(t+h)-\gamma(t)}{h}\Big|\\
			&\geq \lim_{h\to 0}\lim_{n\to\infty}\Big(\Big|\frac{p_{n}(t+h)-p_{n}(t)}{h}\Big|-\Big|\frac{(\gamma(t+h)-p_{n}(t+h))-(\gamma(t)-p_{n}(t))}{h}\Big|\Big)\\
			&=\lim_{h\to 0}\lim_{n\to\infty}\Big|\frac{p_{n}(t+h)-p_{n}(t)}{h}\Big|
			\geq \lim_{h\to 0}\Big|\frac{\eta(t+h)-\eta(t)}{h}\Big|=\abs{\eta'(t)}=1.
		\end{align*}
	\textbf{Step 2}		
		Denote by $p'^{-}$ and $p'^{+}$ the left and right derivative of a polygon. From the curvature bound and Corollary \ref{polygonstaysoutsideofsphere} we know that any sphere
		of curvature $K$ attached tangentially to the direction $p_{n}'^{+}(t)$ at a vertex $p_{n}(t)$, and thus a whole horn torus, cannot contain any vertex of $p_{n}$ restricted to
		$(t,t+\frac{\pi}{2K})$, and the same is true for $p_{n}'^{-}(t)$ with regard to $(t-\frac{\pi}{2K},t)$. Let 
		\begin{align}\label{convergingts}
			t_{n_{k}}\to t\quad\text{such that}\quad p_{n_{k}}(t_{n_{k}})\text{ is a vertex}\quad\text{and}\quad p_{n_{k}}'^{\pm}(t_{n_{k}})\to u^{\pm}\in \sphere^{d-1}.
		\end{align}
		Then $u^{+}=u^{-}$ since
		\begin{align*}
			\MoveEqLeft
			d(u^{+},u^{-})\leq d(u^{+},p_{n_{k}}'^{+}(t_{n_{k}}))+d(p_{n_{k}}'^{+}(t_{n_{k}}),p_{n_{k}}'^{-}(t_{n_{k}}))+d(p_{n_{k}}'^{-}(t_{n_{k}}),u^{-})\\
			&\leq d(u^{+},p_{n_{k}}'^{+}(t_{n_{k}}))+\frac{K}{n_{k}}+d(p_{n_{k}}'^{-}(t_{n_{k}}),u^{-})\to 0.
		\end{align*}
		For every $t$ we can find a sequence of $t_{n_{k}}$ with \eqref{convergingts} and thanks to $p_{n_{k}}\to\gamma$ in $C^{0}$ the (two) horn tori belonging to $p_{n_{k}}(t_{n_{k}})$ 
		converge to a horn torus at $\gamma(t)$ in direction $u^{+}=u^{-}$ such that $\gamma$ does not enter the torus on the parameter range $B_{\frac{\pi}{4K}}(t)$. 
		Then according to \cite[Satz 2.14, p.26]{Gerlach2004a} holds $\gamma\in C^{1,1}(\sphere_{1},\R^{d})$ and $\maxCurv(\gamma)\leq K$. 
		Especially,	$\gamma'(t)=u^{\pm}$.\\
	\textbf{Step 3}
		If we had $||p_{n}'-\gamma'||_{L^{\infty}}\to 0$ then for every $\epsilon>0$ there is an $N$ such that for $n\geq N$ holds
		\begin{align*}
			\abs{p_{n}'^{+}(\tfrac{i}{n})-\gamma'(\tfrac{i}{n})}=\abs{p_{n}'(t)-\gamma'(\tfrac{i}{n})}\leq \abs{p_{n}'(t)-\gamma'(t)}+\abs{\gamma'(t)-\gamma'(\tfrac{i}{n})}
			\leq \epsilon +\frac{K}{n}
		\end{align*}
		for all $i\in\{0,\ldots,n-1\}$, $t\in (\frac{i}{n},\frac{i+1}{n})$. Hence,
		\begin{align}\label{uniformconvergenceassumption}
			\sup_{i=0,\ldots,n-1}\abs{p_{n}'^{+}(\tfrac{i}{n})-\gamma'(\tfrac{i}{n})}\xrightarrow[n\to\infty]{}0.
		\end{align}
		If on the other hand \eqref{uniformconvergenceassumption} holds then for every $t$ where $p_{n}(t)$ is not a vertex we find $i=i(n)$ and for every $\epsilon>0$ an $N$ such that for $n\geq N$ one has
		\begin{align*}
			\abs{p_{n}'(t)-\gamma'(t)}\leq \abs{p_{n}'^{+}(\tfrac{i}{n})-\gamma'(\tfrac{i}{n})}+\abs{\gamma'(\tfrac{i}{n})-\gamma'(t)}
			\leq \epsilon+\frac{K}{n}\to 0.
		\end{align*}
		Thus, \eqref{uniformconvergenceassumption} is equivalent to $||p_{n}'-\gamma'||_{L^{\infty}}\to 0$.		
		Assume that $||p_{n_{k}}'-\gamma'||_{L^{\infty}}\not\to 0$. Then there is a sequence of parameters $t_{n_{k}}$ as in \eqref{convergingts} with
		$p_{n_{k}}'^{+}(t_{n_{k}})\to u^{+}\not=\gamma'(t)$, which contradicts the results of Step 1. 
		Hence $p_{n_{k}}\to \gamma$ in $W^{1,\infty}$.
\end{proof}

\begin{proposition4}[Compactness II]
	Let $p_{n}\in \mathcal{P}_{n}(\mathcal{K})$ bounded in $L^{\infty}$ with $\liminf_{n\to\infty}\Delta_{n}[p_{n}]^{-1}<\infty$. 
	Then there is
	\begin{align*}
		\gamma\in \mathcal{C}(\mathcal{K})\cap C^{1,1}(\sphere_{1},\R^{d})\quad\text{with}\quad p_{n_{k}}\to \gamma\text{ in }W^{1,\infty}(\sphere_{1},\R^{d}).
	\end{align*}
\end{proposition4}
\begin{proof}
	Without loss of generality let $\Delta_{n}[p_{n}]^{-1}\leq K<\infty$ for all $n\in\N$. Note, that $\Delta_{n}[p_{n}]^{-1}<\infty$ means that $p_{n}$ is injective.
	From Proposition \ref{compactness} we know that there is a subsequence
	converging to $\gamma\in \mathcal{C}\cap C^{1,1}(\sphere_{1},\R^{d})$ in $W^{1,\infty}(\sphere_{1},\R^{d})$. It remains to be shown that $\gamma\in\mathcal{K}$.
	In order to deduce this from Proposition \ref{polygonalconvergencedoesnotchangeknotclass} we must show that $\gamma$ is injective. Assume that this is not the case.
	Then there are $s\not=t$ with $\gamma(s)=\gamma(t)$=x. Let $r_{n}\vcentcolon=||\gamma-p_{n}||_{L^{\infty}(\sphere_{1},\R^{d})}+\frac{1}{n}$, i.e., $p_{n}(s),p_{n}(t)\in B_{r_{n}}(x)$,
	and let $n$ be large enough to be sure that there are $u,v$ with $p_{n}(u),p_{n}(v)\not\in B_{4r_{n}}(x)$. The \emph{singly critical self distance}
	$\scsd(p)$ of a polygon $p$ is given by $\scsd(p)\vcentcolon=\min_{(y,z)\in \crit(p)}\abs{z-y}$, where $\crit(p)$ consists of pairs $(y,z)$ where $y=p(t)$ and $z=p(s)$ and  $s$ locally extremizes 
	$w\mapsto\abs{p(t)-p(w)}^{2}$. In \cite[Theorem 3.6]{Millett2008a} it was shown that for $p\in\mathcal{P}_{n}$ holds $\Delta_{n}[p]=\min\{\minRad(p),\scsd(p)\}$. Since the mapping 
	$f(w)=\abs{p_{n}(t)-p_{n}(w)}$ is continuous with $f(s)\leq 2r_{n}$ and $f(u),f(v)\geq 3r_{n}$ we have
	\begin{align*}
		\scsd(p_{n})\leq \min_{\alpha}f\leq f(s)=\abs{p_{n}(t)-p_{n}(s)}\leq 2r_{n}\to 0,
	\end{align*}
	where $\alpha$ is the arc on $\sphere_{1}$ from $u$ to $v$ that contains $s$. This contradicts $\Delta_{n}[p_{n}]^{-1}\leq K$. Thus, we have proven the proposition.
\end{proof}

\begin{proposition}[Convergence of polygons does not change knot class]\label{polygonalconvergencedoesnotchangeknotclass}
	Let $\gamma\in \mathcal{C}\cap C^{1,1}(\sphere_{1},\R^{d})$ be injective and $p_{n}\in \mathcal{P}_{n}(\mathcal{K})$ with $p_{n}\to \gamma$ in $W^{1,\infty}$. Then $\gamma\in \mathcal{K}$.
\end{proposition}
\begin{proof}
	\textbf{Step 1}
		For $||p-\gamma||_{W^{1,\infty}}\leq \frac{\Delta[\gamma]}{2}$ \cite[Lemma 4]{Gonzalez2003a} together with Lemma \ref{globalbiLipschitz} and \cite[4.8 Theorem (8)]{Federer1959a} allows us to estimate
		\begin{align}\label{estimateinversegammafederer}
			\abs{\gamma^{-1}(\xi_{\gamma}(\gamma(s)))-\gamma^{-1}(\xi_{\gamma}(p(s)))}\leq \tilde c^{-1}\abs{\xi_{\gamma}(\gamma(s))-\xi_{\gamma}(p(s))}
			\leq 2\tilde c^{-1}\abs{\gamma(s)-p(s)}.
		\end{align}
		Here, $\xi_{\gamma}$ is the nearest point projection onto $\gamma$. This means
		\begin{align}\label{estimatepolygoncurvenpp}
			\begin{split}
				\MoveEqLeft
				\abs{p'(s)-\gamma'(\gamma^{-1}(\xi_{\gamma}(p(s))))}
				\leq \abs{p'(s)-\gamma'(s)}+\abs{\gamma'(s)-\gamma'(\gamma^{-1}(\xi_{\gamma}(p(s))))}\\
				&\leq ||p'-\gamma'||_{L^{\infty}}+\Delta[\gamma]^{-1}\abs{s-\gamma^{-1}(\xi_{\gamma}(p(s)))}\\
				&=||p'-\gamma'||_{L^{\infty}}+\Delta[\gamma]^{-1}\abs{\gamma^{-1}(\xi_{\gamma}(\gamma(s)))-\gamma^{-1}(\xi_{\gamma}(p(s)))}\\
				&\leq ||p'-\gamma'||_{L^{\infty}}+\Delta[\gamma]^{-1}2\tilde c^{-1}\abs{\gamma(s)-p(s)}
				\leq C||p-\gamma||_{W^{1,\infty}}.
			\end{split}
		\end{align}
		Note that although we have a fixed parameter $s$ we still can estimate $\abs{p'(s)-\gamma'(s)}\leq ||p'-\gamma'||_{L^{\infty}}$ since $p'-\gamma'$ is piecewise continuous.
		If $p(s)$ is a vertex the estimate still holds if we identify $p'(s)$ with either the left or right derivative.\\
	\textbf{Step 2}
		Let $s_{n},t_{n}\in I$, $s_{n}<t_{n}$ with $\xi_{\gamma}(p_{n}(s_{n}))=\xi_{\gamma}(p_{n}(t_{n}))$. We want to show that this situation can only happen for a finite number of $n$.
		Assume that this is not true.
		Let $u_{n}\in [s_{n},t_{n}]$ such that $p_{n}(u_{n})$ is a vertex and maximizes the distance to $\gamma(y_{n})+\gamma'(y_{n})^{\perp}$ for $y_{n}=\gamma^{-1}(\xi_{\gamma}(p(s_{n})))$. 
		For the right derivative $p'^{+}(u_{n})$ holds $d(p'^{+}(u_{n}),\gamma'(y_{n}))\geq \frac{\pi}{2}$.
		\begin{figure}
			\centering 
			\includegraphics[width=0.7\textwidth]{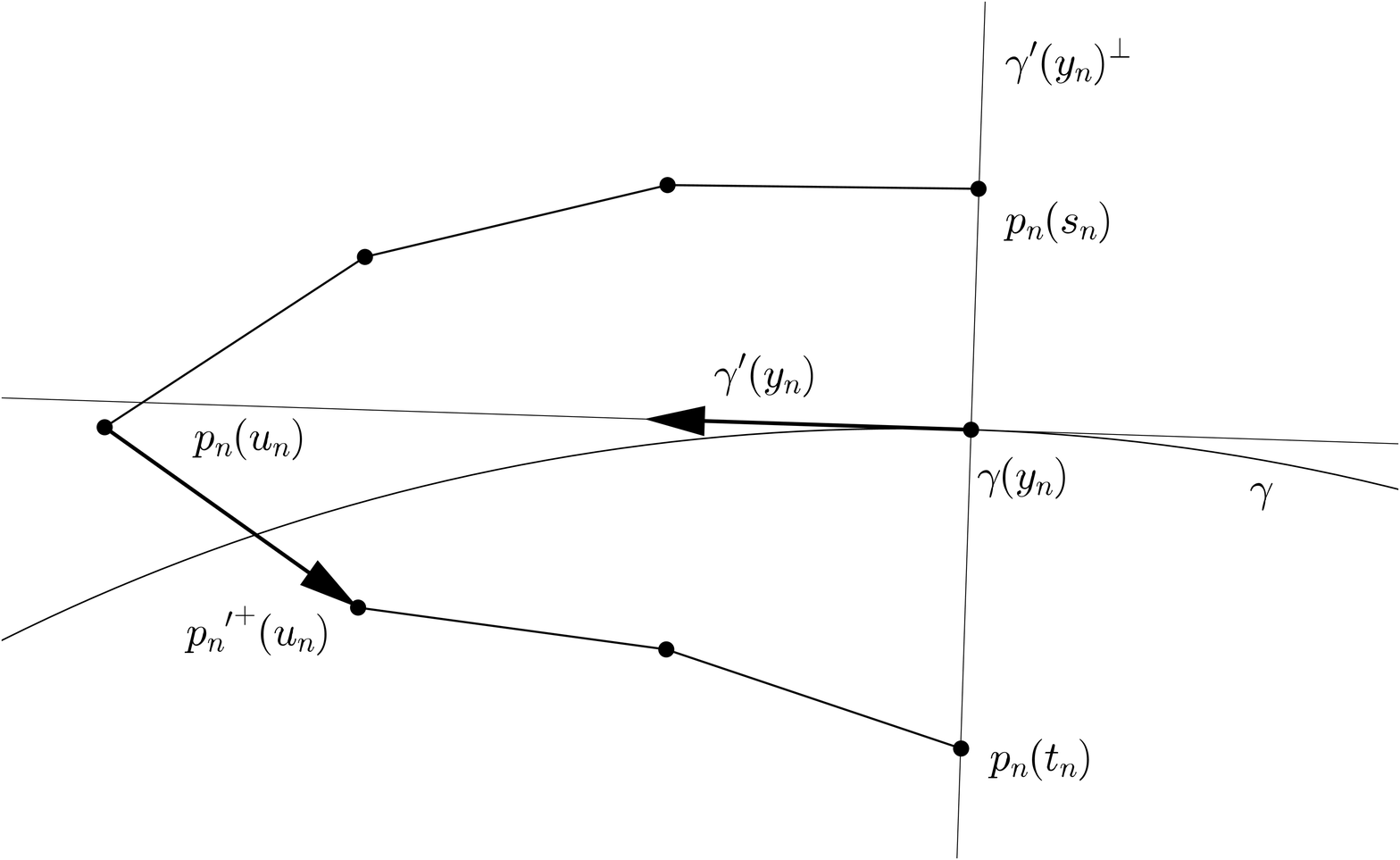}
			\caption{The situation in the proof of Proposition \ref{polygonalconvergencedoesnotchangeknotclass}.}
			\label{picturedistancenormalspace}
		\end{figure}
		As in \eqref{estimateinversegammafederer} we have $\abs{p_{n}(s_{n})-p_{n}(t_{n})}\leq 4\tilde c^{-1}||p_{n}-\gamma||_{W^{1,\infty}}$ 
		and hence for some subsequence $s_{n}\to s_{0}$, $t_{n}\to t_{0}$ and $p_{n}(s_{n})\to \gamma(s_{0})$, $p_{n}(t_{n})\to\gamma(t_{0})$ so that $s_{0}=t_{0}$, since $\gamma$ is injective.
		Therefore also $p_{n}(u_{n})\to \gamma(t_{0})$. But on the other hand \eqref{estimatepolygoncurvenpp} for $s=u_{n}$, $\gamma^{-1}(\xi_{\gamma}(p_{n}(u_{n})))=z_{n}$ 
		and $d$ the distance on the sphere gives a contradiction via
		\begin{align*}
			\MoveEqLeft
			\frac{\pi}{2}-\frac{\pi}{2}C||p_{n}-\gamma||_{W^{1,\infty}}\stackrel{\text{\eqref{estimatepolygoncurvenpp}}}{\leq} d(p'^{+}(u_{n}),\gamma'(y_{n}))-d(p'^{+}(u_{n}),\gamma'(z_{n}))\\
			&\leq d(\gamma'(y_{n}),\gamma'(z_{n}))\leq \frac{\pi}{2}\Delta[\gamma]^{-1}\abs{y_{n}-z_{n}}
			\stackrel{\text{\eqref{estimateinversegammafederer}}}{\leq} \frac{\pi}{2}\Delta[\gamma]^{-1} 2\tilde c^{-1}\abs{p_{n}(s_{n})-p_{n}(u_{n})}\rightarrow 0.
		\end{align*}
	\textbf{Step 3}
		Now we are in a situation similar to \cite[Proof of Lemma 5]{Gonzalez2002b}, \cite[Theorem 4.10]{Strzelecki2013b} and as there we can construct an ambient isotopy by 
		moving the point $p_{n}(s)$ to $\gamma(\gamma^{-1}(\xi_{\gamma}(p_{n}(s))))$ 
		along a straight line segment in the circular cross section of the tubular neighbourhood about $\gamma$.
\end{proof}

\begin{lemma}[Injective locally bi-L. mappings on compact sets are globally bi-L.]\label{globalbiLipschitz}
	Let $(K,d_{1})$, $(X,d_{2})$ be non-empty metric spaces, $K$ compact and $f:K\to X$ be an injective mapping that is locally bi-Lipschitz, i.e.,
	there are constants $c,C>0$ such that for every $x\in K$ there is a neighbourhood $U_{x}$ of $x$ with
	\begin{align*}
		c\, d_{1}(x,y)\leq d_{2}(f(x),f(y))\leq C d_{1}(x,y)\quad\text{for all }y\in U_{x}.
	\end{align*}
	Then there are constants $\tilde c,\tilde C>0$ with
	\begin{align}\label{globalbiLipschitzeq}
		\tilde c\, d_{1}(x,y)\leq d_{2}(f(x),f(y))\leq \tilde C d_{1}(x,y) \quad\text{for all }x,y\in K.
	\end{align}
\end{lemma}
\begin{proof}
	By Lebesgue's Covering Lemma we obtain a $\mathrm{diam}(K)>\delta>0$ such that $(B_{\delta}(x))_{x\in K}$ is a refinement of $(U_{x})_{x\in K}$. 
	Then $K_{\delta}\vcentcolon=\{(x,y)\in K^{2}\mid d_{1}(x,y)\geq \delta\}$
	is compact and non-empty. Hence
	\begin{align*}
		0<\epsilon\vcentcolon= \min_{(x,y)\in K_{\delta}}d_{2}(f(x),f(y))\leq \max_{(x,y)\in K_{\delta}}d_{2}(f(x),f(y))=\vcentcolon M<\infty,
	\end{align*}
	since $\mathrm{diag}(K)\cap K_{\delta}=\emptyset$ and $f$ is continuous and injective. Thus
	\begin{align*}
		d_{2}(f(x),f(y))\leq M = C'\delta\leq C'd_{1}(x,y)\quad\text{for all }x,y\in K_{\delta}
	\end{align*}
	holds for $C'\vcentcolon=M\delta^{-1}$ and
	\begin{align*}
		c'd_{1}(x,y)\leq c'\mathrm{diam}(K)=\epsilon\leq d_{2}(f(x),f(y))\quad\text{for all }x,y\in K_{\delta}
	\end{align*}
	for $c'\vcentcolon=\epsilon \mathrm{diam}(K)^{-1}$. Choosing $\tilde c\vcentcolon=\min\{c,c'\}$ and $\tilde C\vcentcolon=\max\{C,C'\}$ yields \eqref{globalbiLipschitzeq},
	because $(x,y)\not\in K_{\delta}$ implies $y\in B_{\delta}(x)\subset U_{x}$.
\end{proof}

\section{The $\liminf$ inequality }\label{sectionliminf}

Using Schur's Theorem for curves of finite total curvature, see for example \cite[Theorem 5.1]{Sullivan2008a}, we can prove Rawdon's result \cite[Lemma 2.9.7, p.58]{Rawdon1997a} for embedded $C^{1,1}$ curves.
Note, that especially the estimate from \cite[Proof of Theorem 2]{Litherland1999a} that is implicitly used in the proof of \cite[Lemma 2.9.7, p.58]{Rawdon1997a} holds for $C^{1,1}$ curves.

\begin{lemma}[Approximation of curves with $\frac{\dcsd(\gamma)}{2}<\minRad(\gamma)$]\label{approximationcurveswithsmalldcsc}
	Let $\gamma\in \mathcal{C}(\mathcal{K})\cap C^{1,1}(\sphere_{1},\R^{d})$ and $p\in\mathcal{P}_{n}$ for some $n$ such that
	\begin{align*}
		\minRad(\gamma)-\frac{\dcsd(\gamma)}{2}=\delta>0\quad\text{and}\quad
		||\gamma-p||_{L^{\infty}}< \epsilon
	\end{align*}
	for $\epsilon<\delta/4$. Then
	\begin{align*}
		\dcsd(p)\leq \dcsd(\gamma)+2\epsilon.
	\end{align*}
\end{lemma}
\begin{proof}
	Let $\minRad(\gamma)-\frac{\dcsd(\gamma)}{2}=\delta>0$, $\epsilon<\delta/4$ and set $d\vcentcolon=\frac{1}{2}(\minRad(\gamma)+\frac{\dcsd(\gamma)}{2})$.
	By \cite[Lemma 2.9.7 2., p.58]{Rawdon1997a} there are $(s_{0},t_{0})\in \overline A_{\pi d}^{\gamma}\vcentcolon=\{ (s,t)\mid d(s,t)\geq \pi d\}$, see notation in \cite{Rawdon1997a},
	such that
	\begin{align*}
		\abs{p(s_{0})-p(t_{0})}<\dcsd(\gamma)+2\epsilon.
	\end{align*}
	Now, let $(\overline s,\overline t)\in \overline A_{\pi d}^{\gamma}$ such that 
	\begin{align}\label{estimateoverlinestsmall}
		\abs{p(\overline s)-p(\overline t)}=\min_{(s,t)\in \overline A_{\pi d}^{\gamma}}\abs{p(s)-p(t)}\leq \abs{p(s_{0})-p(t_{0})}<\dcsd(\gamma)+2\epsilon.
	\end{align}
	Then either $(\overline s,\overline t)$ lie in the open set $A_{\pi d}^{\gamma}\vcentcolon=\{(s,t)\mid d(s, t)>\pi d\}$ or by \cite[Lemma 2.9.7 1., p.58]{Rawdon1997a} holds
	\begin{align*}
		\abs{p(\overline t)-p(\overline s)}\geq \minRad(\gamma)+\frac{\dcsd(\gamma)}{2}-2\epsilon
		= \dcsd(\gamma)+\delta-2\epsilon>\dcsd(\gamma)+2\epsilon,
	\end{align*}
	which contradicts \eqref{estimateoverlinestsmall}. Hence $(\overline s,\overline t)$ lie in the open set $A_{\pi d}^{\gamma}$. This means we can use the argument from
	\cite[Lemma 2.9.8, p.60]{Rawdon1997a} to show that $p(\overline s)$ and $p(\overline t)$ are doubly critical for $p$ and therefore
	\begin{align*}
		\dcsd(p)\leq \abs{p(\overline s)-p(\overline t)}\leq \dcsd(\gamma)+2\epsilon.
	\end{align*}
\end{proof}

\begin{proposition}[The $\liminf$ inequality]\label{liminfinequality}
	Let $\gamma\in \mathcal{C}(\mathcal{K})$, $p_{n}\in \mathcal{P}_{n}(\mathcal{K})$ with $p_{n}\to \gamma$ in $W^{1,\infty}$ for $n\to\infty$. Then
	\begin{align*}
		\Delta[\gamma]^{-1}\leq \liminf_{n\to\infty}\Delta_{n}[p_{n}]^{-1}.
	\end{align*}
\end{proposition}
\begin{proof}
	By Proposition \ref{compactnessdcsd} we might assume without loss of generality that $\gamma\in C^{1,1}(\sphere_{1},\R^{d})$.
	In case $\Delta[\gamma]^{-1}=\maxCurv(\gamma)$ the proposition follows from Proposition \ref{compactness} and in case $\Delta[\gamma]^{-1}=\frac{2}{\dcsd(\gamma)}>\maxCurv(\gamma)$
	Lemma \ref{approximationcurveswithsmalldcsc} gives $\limsup_{n\to\infty}\dcsd(p_{n})\leq \dcsd(\gamma)$, so that 
	\begin{align*}
		\Delta[\gamma]^{-1}=\frac{2}{\dcsd(\gamma)}\leq \liminf_{n\to\infty}\frac{2}{\dcsd(p_{n})}\leq \liminf_{n\to\infty}\Delta_{n}[p_{n}]^{-1}.
	\end{align*}
\end{proof}

Clearly, the previous proposition also holds for subsequences $p_{n_{k}}$.

\section{The $\limsup$ inequality }\label{sectionlimsup}

\begin{proposition}[The $\limsup$ inequality]\label{limsupinequality}
	For every $\gamma\in \mathcal{C}(\mathcal{K})\cap C^{1,1}(\sphere_{1},\R^{d})$ there are $p_{n}\in \mathcal{P}_{n}(\mathcal{K})$ with $p_{n}\to \gamma$ in $W^{1,\infty}$ and
	\begin{align*}
		\limsup_{n\to\infty}\Delta_{n}[p_{n}]^{-1}\leq \Delta[\gamma]^{-1}.
	\end{align*}
\end{proposition}
\begin{proof}
	In \cite[Proposition 10]{Scholtes2013d} we showed that if $n$ is large enough we can find an equilateral inscribed closed polygon $\tilde p_{n}$ of length $\tilde L_{n}\leq 1$ with $n$ vertices that lies 
	in the same knot class as $\gamma$. By rescaling it to unit length via $p_{n}(t)=L\tilde L_{n}^{-1}\tilde p_{n}(\tilde L_{n} L^{-1}t)$, $L=1$, we could show in addition
	that $p_{n}\to\gamma$ in $W^{1,2}(\sphere_{1},\R^{d})$, as
	$n\to\infty$. It is easily seen, exploiting $\gamma'$ Lipschitz, that for $\gamma\in C^{1,1}(\sphere_{1},\R^{d})$ this convergence can be improved to convergence in $||\cdot||_{W^{1,\infty}(\sphere_{1},\R^{d})}$.\\
	\textbf{Step 1}
		From Figure \ref{discretecurvaturepicture} we see $r=r(x,y,z)$ and
		\begin{align*}
			\kappa_{d}(x,y,z)=\frac{4\tan(\frac{\phi}{2})}{\abs{x-y}+\abs{z-y}}=\frac{2\tan(\frac{\alpha+\beta}{2})}{\sin(\alpha)+\sin(\beta)}\,\frac{1}{r}.
		\end{align*}
		\begin{figure}
			\centering 
			\includegraphics[width=0.5\textwidth]{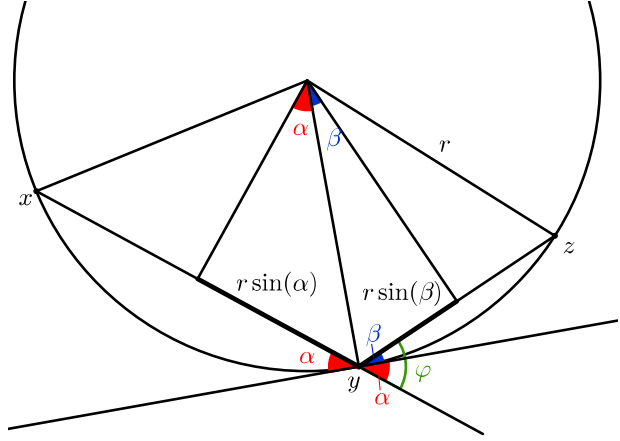}
			\caption{Quantities for the computation of discrete curvature.}
			\label{discretecurvaturepicture}
		\end{figure}%
		Thus, we can estimate
		\begin{align}
			\begin{split}\label{longerangleestimate}
				0\leq \frac{2\tan(\frac{\alpha+\beta}{2})}{\sin(\alpha)+\sin(\beta)}-1\leq \frac{\tan(\alpha+\beta)}{\sin(\alpha)+\sin(\beta)}-1
				\leq \frac{\tan(\alpha+\beta)}{\sin(\alpha+\beta)}-1=\frac{1-\cos(\alpha+\beta)}{\cos(\alpha+\beta)}
				\leq (\alpha+\beta)^{2}
			\end{split}
		\end{align}
		for $\alpha,\beta\in [0,\frac{\pi}{6}]$, since 
		\begin{align*}
			\MoveEqLeft
			\sin(\alpha)+\sin(\beta)=2\Big(\sin(\tfrac{\alpha}{2})\cos(\tfrac{\alpha}{2})+\sin(\tfrac{\beta}{2})\cos(\tfrac{\beta}{2})\Big)
			\leq 2\Big(\sin(\tfrac{\alpha}{2})+\sin(\tfrac{\beta}{2})\Big)\\
			&\leq 2\frac{\sin(\frac{\alpha}{2})\cos(\frac{\beta}{2})+\sin(\frac{\beta}{2})\cos(\frac{\alpha}{2})}{\cos(\frac{\alpha+\beta}{2})} =  2\tan(\tfrac{\alpha+\beta}{2}),
		\end{align*}
		$2\tan(\frac{x}{2})\leq \frac{2\tan(\frac{x}{2})}{1-\tan^{2}(\frac{x}{2})}=\tan(x)$ and $\frac{1}{2}\leq \cos(\alpha+\beta)$, as well as $1-\frac{(\alpha+\beta)^{2}}{2}\leq \cos(\alpha+\beta)$.
		Let $x=\gamma(s)$, $y=\gamma(t)$ and $z=\gamma(u)$ for $s<t<u$ with $\abs{t-s},\abs{u-t}\leq\frac{2L}{n}$. Now, again by Figure \ref{discretecurvaturepicture}, we have
		\begin{align*}
			2Ln^{-1}\geq\abs{t-s}\geq \abs{y-x}=2\sin(\alpha)r\geq 4\pi^{-1}\alpha r\geq 4\pi^{-1}\alpha \Delta[\gamma]\geq \alpha\Delta[\gamma],
		\end{align*}
		or in other words $\alpha\leq 2L\Delta[\gamma]^{-1}n^{-1}$ and the same is true for $\beta$. According to \eqref{longerangleestimate} we can estimate
		\begin{align*}
			\kappa_{d}(x,y,z)\leq \frac{1+(\alpha+\beta)^{2}}{r}\leq (1+16L^{2}\Delta[\gamma]^{-2}n^{-2})\Delta[\gamma]^{-1}.
		\end{align*}
		This means for the sequence of inscribed polygons $\tilde p_{n}$ that
		\begin{align*}
			\limsup_{n\to\infty}\maxCurv(\tilde p_{n})\leq \Delta[\gamma]^{-1}.
		\end{align*}
	\textbf{Step 2}
		According to \cite[Lemma 2.8.2, p.46]{Rawdon1997a} the total curvature between two doubly critical points of polygons must be at least $\pi$. Let $\tilde p_{n}(s_{n})$ and $\tilde p_{n}(t_{n})$ be doubly critical
		for $p_{n}$. Using the curvature bound from the previous step we obtain $\pi\leq 2\Delta[\gamma]^{-1}\abs{t_{n}-s_{n}}$, so that $s_{n}$ and $t_{n}$ cannot converge to
		the same limit. From Lemma \ref{limitsofdoublycriticalpoints} we directly obtain
		\begin{align*}
			\dcsd(\gamma)\leq \liminf_{n\to\infty}\dcsd(\tilde p_{n})\quad\Rightarrow\quad
			\limsup_{n\to\infty}\frac{2}{\dcsd(\tilde p_{n})}\leq \frac{2}{\dcsd(\gamma)}\leq\Delta[\gamma]^{-1} .
		\end{align*}
	\textbf{Step 3}
		Noting that $L\tilde L_{n}^{-1}\to 1$ the previous steps yield
		\begin{align*}
			\limsup_{n\to\infty}\Delta_{n}[p_{n}]^{-1}=\limsup_{n\to\infty}\max\Big\{\maxCurv(p_{n}), \frac{2}{\dcsd(p_{n})}\Big\}
			\leq \Delta[\gamma]^{-1}.
		\end{align*}	
\end{proof}

\begin{lemma}[Limits of double critical points are double critical]\label{limitsofdoublycriticalpoints}
	Let $\gamma\in \mathcal{C}(\mathcal{K})\cap C^{1,1}(\sphere_{1},\R^{d})$, $p_{n}\in \mathcal{P}_{n}$ with $p_{n}\to \gamma$ in $W^{1,\infty}(\sphere_{1},\R^{d})$. 
	Let $s_{n}\not=t_{n}$ be such that $s_{n}\to s$, $t_{n}\to t$ and $s\not= t$.
	If $p_{n}(s_{n})$ and $p_{n}(t_{n})$ are double critical for $p_{n}$. Then $\gamma(s)$ and $\gamma(t)$ are double critical for $\gamma$.
\end{lemma}
\begin{proof}
	Denote by $p'^{+}$ and $p'^{-}$ the right and left derivative of a polygon $p$. Since the piecewise continuous derivatives $p_{n}'$ converge in $L^{\infty}$ to the continuous derivatives $\gamma$
	we have
	\begin{align*}
		0\geq \langle p_{n}'^{+}(s_{n}), p_{n}(t_{n})-p(s_{n})\rangle\cdot\langle p_{n}'^{-}(s_{n}), p_{n}(t_{n})-p(s_{n}) \rangle\to \langle \gamma'(s), \gamma(t)-\gamma(s)\rangle^{2}.
	\end{align*}
	The analogous result is obtained if we change the roles of $s$ and $t$, so that $\gamma(t)$ and $\gamma(s)$ are double critical for $\gamma$.
\end{proof}

\section{Discrete Minimizers}\label{sectiondiscreteminimizers}

\begin{lemma}[Computation of $\Delta_{n}$ for regular $n$-gon $g_{n}$]
	For $n\geq 3$ holds
	\begin{align*}
		\frac{1}{\Delta_{n}[g_{n}]}=2n\tan(\tfrac{\pi}{n}).
	\end{align*}
\end{lemma}
\begin{proof}
	From Figure \ref{regularngons} we see that for the regular $n$-gon $g_{n}$ of length $1$ holds
	\begin{align*}
		\dcsd(g_{n})\geq \frac{1}{n\tan(\frac{\pi}{n})}
	\end{align*}
	\begin{figure}[h!]
		\centering
		\includegraphics[width=0.4\textwidth]{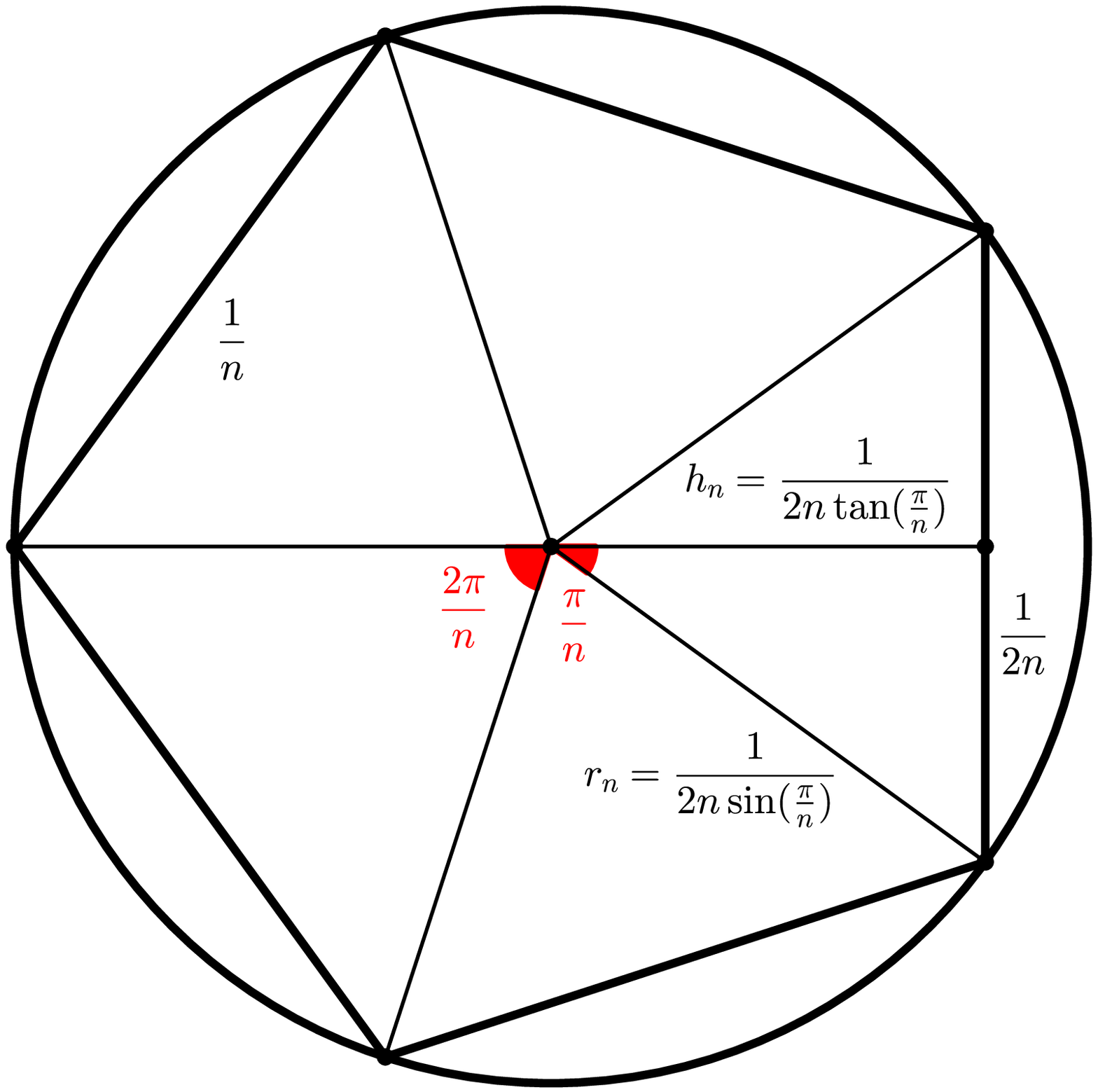}\qquad
		\includegraphics[width=0.4\textwidth]{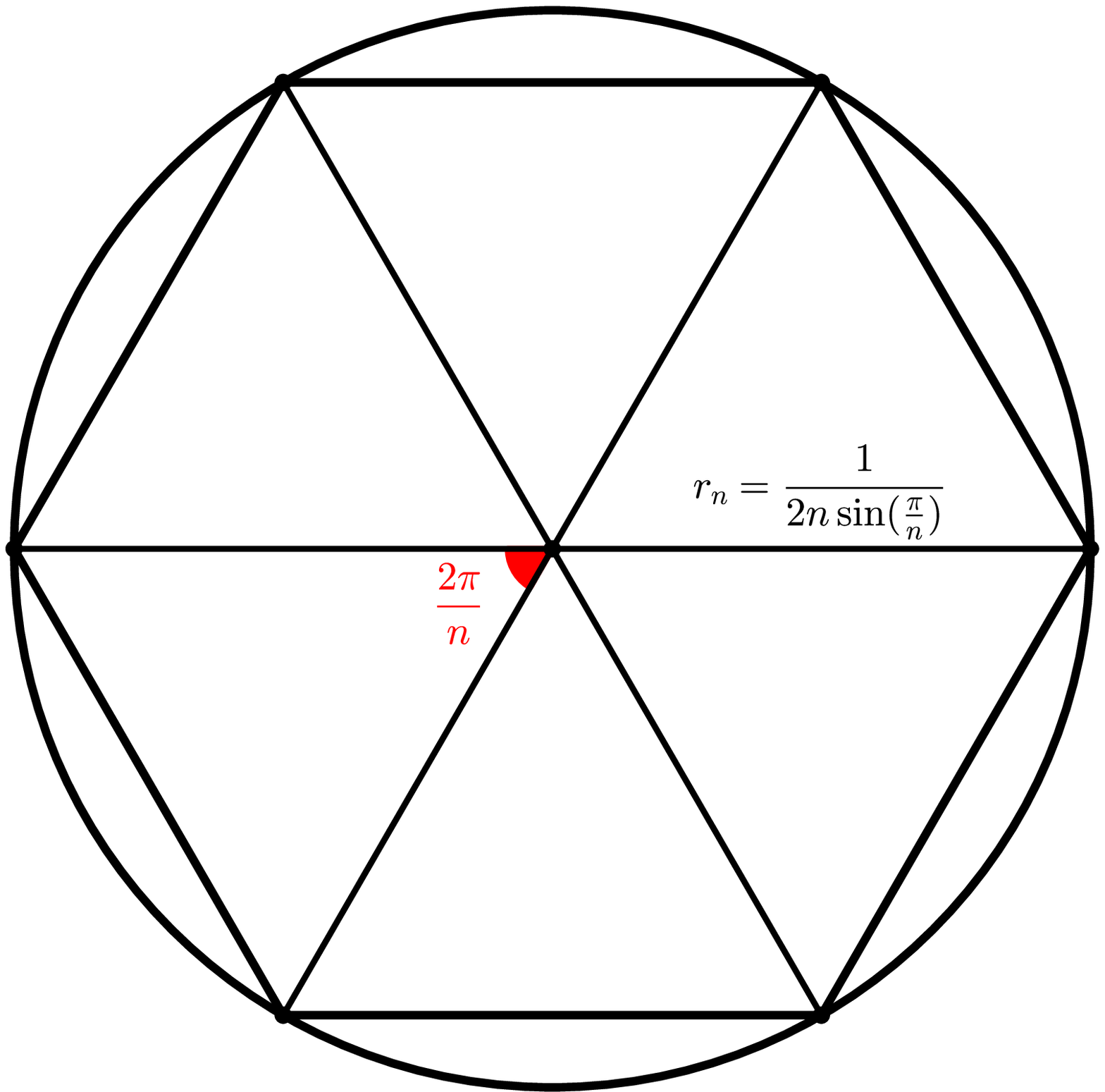}
		\caption{Computation of $\dcsd$ for regular $n$-gons of length $1$.}
		\label{regularngons}
	\end{figure}
	and as $\maxCurv(g_{n})=2n\tan(\frac{\pi}{n})$ by Figure \ref{discretecurvaturepicture} we have shown the proposition.
\end{proof}

\begin{proposition5}[Regular $n$-gon is unique minimizer of $\Delta_{n}^{-1}$]
	Let $p\in \mathcal{P}_{n}$ then
	\begin{align*}
		\Delta_{n}[g_{n}]^{-1}\leq \Delta_{n}[p]^{-1},
	\end{align*}
	with equality if and only if $p$ is a regular $n$-gon.
\end{proposition5}
\begin{proof}
	According to Fenchel's Theorem for polygons, see \cite[3.4 Theorem]{Milnor1950a}, the total curvature is at least $2\pi$, 
	i.e., $\sum_{i=1}^{n}\phi_{i}\geq 2\pi$ for the exterior angles $\phi_{i}=\measuredangle(x_{i}-x_{i-1},x_{i+1}-x_{i})$. 
	This means there must be $j\in\{1,\ldots,n\}$ with $\phi_{j}\geq \frac{2\pi}{n}$. Thus 
	\begin{align}\label{inequalityregularngonthickness}
		\Delta_{n}[p]^{-1}\geq \maxCurv(p)\geq 2n\tan(\tfrac{\phi_{j}}{2})\geq 2n\tan(\tfrac{\pi}{n})=\Delta_{n}[g_{n}]^{-1}.
	\end{align}
	Equality holds in Fenchel's Theorem if and only if $p$ is a convex planar curve. If $\phi_{j}<\frac{2\pi}{n}$ there must be $\phi_{k}>\frac{2\pi}{n}$ 
	and thus $\Delta_{n}[p]^{-1}>\Delta_{n}[g_{n}]^{-1}$. Since the regular $n$-gon $g_{n}$ is the only convex equilateral polygon with $\phi_{i}=\frac{2\pi}{n}$ for $i=1,\ldots,n$
	we have equality in \eqref{inequalityregularngonthickness} if and only if $p$ is a regular $n$-gon.
\end{proof}

\setlength{\parindent}{0cm}

\bibliography{/Users/sebastianscholtes/Documents/library.bib}{}
\bibliographystyle{abbrv}

\end{document}